\documentclass[runningheads]{llncs}
\usepackage{graphicx}
\usepackage{epsfig}
\usepackage{graphics}
\usepackage{array}
\usepackage[superscript]{cite}
\usepackage{amsmath}
\usepackage{algorithm}
\usepackage{algorithmic}
\usepackage{epstopdf}
\usepackage{amssymb}
\usepackage{fullpage}
\date{}

\newtheorem{dfn}{Definition}
\title{On $2K_2$-free graphs - Structural and Combinatorial View}
\author{S.Dhanalakshmi, N.Sadagopan and V.Manogna} 
\institute{Indian Institute of Information Technology, Design and Manufacturing, Kancheepuram, Chennai, India. \\
\email{$\{mat12d001, sadagopan\}@iiitdm.ac.in$}}
\begin{document}
\maketitle

\begin{abstract}

 A connected graph is $2K_2$-free if it does not contain a pair of independent edges as an induced subgraph. In this paper,  we present the structural characterization of minimal vertex separator and show that there are polynomial number of minimal vertex separators in $2K_2$-free graphs. Further, using the enumeration we show that finding minimum connected vertex separator in $2K_2$-free graphs is polynomial time solvable. We highlight that finding minimum connected vertex separator is NP-complete in Chordality 5 graphs, which is a super graph class of $2K_2$-free graphs. Other study includes, enumeration of all distinct maximal independent sets and testing $2K_2$-free graphs. Also, we present an polynomial time algorithm for feedback vertex set problem in the subclass of $2K_2$-free graphs. 

\end{abstract}

\section{Introduction}
The study of graphs with forbidden subgraphs has attracted researchers from both mathematics and computing.   The popular ones are chordal graphs \cite{dirac}, chordal bipartite graphs \cite{fulkerson}, etc.  Largely, these graphs were discovered to study the gap between NP-completeness and polynomial-time solvability of a combinatorial problem.  On the similar line, special graph classes like   $2K_2$-free graphs \cite{erdos}, planar graphs \cite{mozes}, interval graphs \cite{pavol}, circular-arc graphs \cite{stahl}, etc., were discovered in the literature. In particular, $2K_2$-free graphs are well studied in the literature \cite{chung,dbwest,hujter,lozin,meister,patel} as it contains split graphs \cite{hammer} and co-chordal graphs \cite{golumbic} as its proper subgraph classes. In this paper, we revisit $2K_2$-free graphs and investigate from structural and combinatorial perspectives.  This line of study has been considered important in the literature as structural observations resulting from this study may yield polynomial-time algorithms for well-known combinatorial problems such as vertex cover, coloring, etc.  \\

\noindent While $2K_2$-free graphs received good attention in the past, the fundamental problems like structural characterization using minimal vertex separator has not been addressed in the literature. This line of study is fruitful as it yields efficient algorithm for testing $2K_2$-free graphs. Like chordal graphs \cite{dirac}, we also give a structural characterization in terms of minimal vertex separators.  Further, we show that $2K_2$-free graphs have polynomially many minimal vertex separators and present an algorithm to list all of them.  As a consequence, we show that classical problems related to constrained vertex separators such as finding minimum stable vertex separator and minimum connected vertex separator are polynomial-time solvable restricted to $2K_2$-free graphs. These problems are NP-complete for general graphs. It is important to highlight that finding minimum connected vertex separator is NP-complete in chordality 5 graphs \cite{sadagopan}.  In this paper, we identify the first non-trivial subclass of chordality 5 graphs which are $2K_2$-free graphs and show that finding minimum connected vertex separator is polynomial-time solvable.\\

\noindent Due to its nice structure, it is known from \cite{farber} that $2K_2$-free graphs have polynomial number of maximal independent sets but there does not exist an explicit algorithm in the literature, to list all maximal independent sets in polynomial time. In this paper, we present an algorithm for enumerating all maximal independent sets, which runs in polynomial-time. It is important to highlight that a well-known coloring problem has a polynomial bound in terms of the maximum size of a complete subgraph for $2K_2$-free graphs \cite{wagon} and $3$-colorability problem can be solved in polynomial-time for its super graph class $P_5$-free graphs \cite{shu}. Although, $3$-colorability problem is addressed in the super graph class of $2K_2$-free graphs, we present an alternative algorithm for this problem in terms of MIS. In addition to this, we also shows that a well-known combinatorial problem of finding minimum feedback vertex set can be solved in polynomial-time for the subclass of $2K_2$-free graphs, whereas these problems are NP-complete in general graphs and polynomial when restricted to chordal \cite{cor} and chordal bipartite graphs \cite{ton}. \\

\noindent We use standard graph-theoretic notations as in \cite{golumbic,west}. In particular, we write $V(G)$ and $E(G)$ to denote the vertex set and edge set of a graph $G$. The $neighborhood$ of a vertex $v$ of $G$, $N_G$($v$), is the set of vertices adjacent to $v$ in $G$. The degree of the vertex $v$ is $d_G(v) = \vert N_G(v) \vert$. $\delta(G)$ and $\Delta(G)$ denotes the minimum and maximum degree of a graph $G$, respectively. $P_{uv} = (u=u_1, u_2, \ldots, u_k=v)$ is a \emph{path} defined on $V(P_{uv})=\{u=u_1, u_2, \ldots, u_k=v\}$ such that $E(P_{uv}) = \{\{u_i, u_{i+1}\}\vert \{u_i, u_{i+1}\} \in E(G), 1 \leq i \leq k-1\}$. For simplicity, we use $\vert P_{uv} \vert$ to refer $\vert V(P_{uv}) \vert$. The set $V(P_{uv})\backslash \{u,v\}$ denotes the \emph{internal vertices} of the path $P_{uv}$. $P_n$ denotes the path on $n$ vertices. A \emph{cycle} $C$ on $n$-vertices is denoted as $C_n$, where $V(C) = \{x_1, x_2, \ldots, x_n\}$ and $E(C) = \{\{x_1, x_2\}, \{x_2,x_3\}, \ldots, \{x_{n-1},x_n\}, \{x_n,x_1\}\}$. The graph $G$ is said to be $connected$ if every pair of vertices in $G$ has a path and if the graph is not connected it can be divided into disjoint connected $components$ $G_1, G_2, \ldots, G_k$, $k \geq 2$, where $V(G_i)$ denotes the set of vertices in the component $G_i$. A connected component $G_i$ is trivial if $\vert V(G_i) \vert = 1$. It is non-trivial if $\vert V(G_i) \vert \geq 2$. \\

\noindent The graph $M$ is called a $subgraph$ of $G$ if $V(M)$ $\subseteq$ $V(G)$ and $E(M)\subseteq E(G)$. The subgraph $M$ of a graph $G$ is said to be $induced$ $subgraph$, if for every pair of vertices $u$ and $v$ of $M$, \{$u,v$\} $\in$ $E(M)$ if and only if \{$u,v$\} $\in$ $E(G)$ and it is denoted by $[M]$. Let $S$ be a non-empty subset of $V(G)$ and let $G\backslash S$ denotes the induced subgraph on $V(G)\backslash S$ vertices. The set $S$ is said to be \emph{independent set (stable set)} if every pair of vertices of $S$ is non-adjacent. The set $S$ is said to be \emph{clique} if every pair of vertices of $S$ is adjacent. For a non-trivial non-complete graph, a subset $R \subset V(G)$ is said to be a \emph{vertex separator} if $G\backslash R$ gives distinct connected components. $R$ is minimal vertex separator if there does not exist a subset $R'\subset R$ such that $R'$ is a vertex separator. For a trivial graph, the graph itself a minimal vertex separator. For a non-trivial complete graph on $n$ vertices, neighborhood of every vertex is a minimal vertex separator. A minimal vertex separator is said to be a stable separator if it forms stable set. A minimal vertex separator is said to be a clique separator if it forms clique. Two edges $e_1 = \{a,b\}$ and $e_2 = \{c,d\}$ are said to induce $2K_2$ if $\{a,c\}$, $\{a,d\}$, $\{b,c\}$, $\{b,d\}$ $\notin E(G)$. A graph is $2K_2$-free if it has no such $e_1$ and $e_2$ as an induced sub graph. It is apparent that if a graph $G$ is $2K_2$-free then the complement of $G$ has no induced cycle of length 4. Let $G$ be a $2K_2$-free graph with the vertex set $V(G)$ and edge set $E(G)$ such that $\vert V(G) \vert = n$ and $\vert E(G) \vert = m$ respectively. We use this notation throughout this paper for analysis purpose.

\section{Structural Characterization of $2K_2$-free graphs}
In this section, we first present an alternative characterization of $2K_2$-free graphs using forbidden induced subgraphs. We use this result in the study of minimal vertex separators and testing $2K_2$-free graphs.

\begin{lemma}
 A connected graph is $2K_2$-free if and only if it forbids $H_{1}$, $H_{2}$ and  $H_{3}$ as an induced subgraphs.
 
\vspace{-0.2 cm}
\begin{figure}
\begin{center}
\includegraphics[scale=0.3]{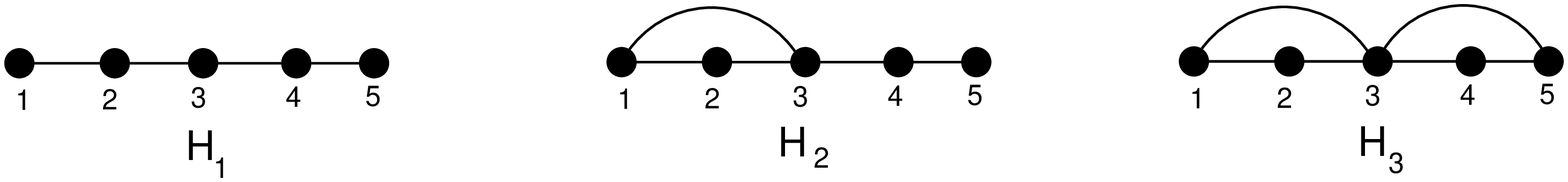} 
\end{center}
\end{figure}
\vspace{-1.4 cm}
\end{lemma}

\begin{proof}
\textbf{Necessity:} On the contrary, assume that $G$ contains $H_{1}$ or $H_{2}$ or $H_{3}$ as an induced subgraph. Clearly, the edges $ \lbrace 1,2 \rbrace$ and  $\lbrace 4,5\rbrace$ form $2K_2$, which is a contradiction. \\
\textbf{Sufficiency: } On the contrary, assume that $G$ is not $2K_2$-free i.e., $G$ contains $2K_2$ as an induced subgraph. let $\lbrace u,v \rbrace$ and $\lbrace x,y \rbrace$ be any two edges in $G$ which induces $2K_2$. Therefore, $\{u,x\},\{u,y\},\{v,x\},\{v,y\} \notin E(G)$. Since $G$ is connected, there exists a shortest path $P_{vx}$ such that $\vert V(P_{vx})\vert \geq 3$.

\begin{figure}
\begin{center}
\includegraphics[scale=0.3]{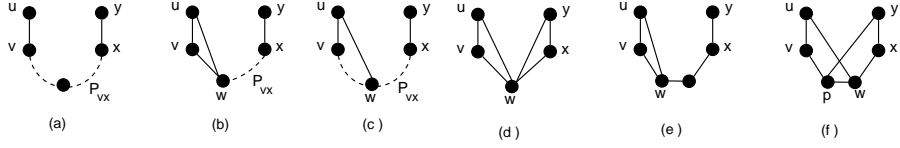}
\caption{Subgraphs which induces $2K_2$}
\end{center}
\label{2k2}
\vspace{-0.5 cm}
\end{figure}

\begin{description}
\item[Case 1:] Neither $u$ is adjacent to an internal vertex in $V(P_{vx})$ nor $y$ is adjacent to an internal vertex in $V(P_{vx})$.\\
Since $\vert V(P_{vx})\vert \geq 3$, $G$ contains induced $P_{n}$, $n \geq 5$ where $P_{n}$=$( u,P_{vx},y)$. Thus, $G$ contains an induced $H_1$ (\emph{see fig.1.(a)}).

\item[Case 2:] Either $u$ is adjacent to an internal vertex in $V(P_{vx})$ or $y$ is adjacent to an internal vertex in $V(P_{vx})$. w.l.o.g. assume that $u$ is adjacent to an internal vertex in $V(P_{vx})$. Since $\vert V(P_{vx})\vert \geq 3$, there exist an internal vertex $w \in V(P_{vx})$ such that $\{u,w\} \in E(G)$. If $\{v,w\} \in E(G)$ then, $G$ has an induced $H_2$ where $(u,v,w)$ forms an induced $C_3$ (\emph{see fig.1.(b)}). If $\{v,w\} \notin E(G)$ then, $G$ contains induced $P_{n}$, $n \geq 5$ where $P_{n}$=$( y,x,P_{xw},P_{wv})$. Thus, $G$ contains an induced $H_1$ (\emph{see fig.1.(c)}).
       
\item[Case 3:] $u$ is adjacent to an internal vertex in $V(P_{vx})$ and $y$ is adjacent to an internal vertex in $V(P_{vx})$.
\begin{description}
        \item[$\bullet$] If $\vert V(P_{vx})\vert = 3$ then, there exist an internal vertex $w \in V(P_{vx})$ such that $\{u,w\}, \{y,w\} \in E(G)$. Thus, $G$ contains an induced $H_{3}$ with $V(H_{3})$=$\{ u,v,w,x,y \}$ where $(u,v,w)$ and $(y,w,x)$ forms an induced $C_3$ (\emph{see fig.1.(d)}). 
        \item[$\bullet$] If $\vert V(P_{vx})\vert = 4$ then, there exist internal vertices $w, p \in V(P_{vx})$ such that $\{u,w\}, \{y,p\} \in E(G)$. If $\{v,w\} \in E(G)$ then, $G$ contains an induced $H_{2}$ with $V(H_{2})$=$\{ u,v\} \cup V(P_{wx})$ where $(u,v,w)$ forms an induced $C_3$ (\emph{see fig.1.(e)}). The argument is similar, if $\{x,p\} \in E(G)$. If neither $\{v,w\} \in E(G)$ nor $\{x,p\} \in E(G)$ then, ($u,v,p,y,x$) forms an induced $H_1$ (\emph{see fig.1.(f)}).
        \item[$\bullet$] If $\vert V(P_{vx})\vert > 4$ then, $G$ contains induced $P_{n}$, $n \geq 5$ with $P_{n}= P_{vx}$. Thus, $G$ contains an induced $H_1$.
\end{description}

\end{description}
In all the above cases $G$ contains $H_1$ or $H_2$ or $H_3$, which is a contradiction. Thus $G$ is a $2K_2$-free graph. $\hfill \qed$
\end{proof}

\begin{dfn}
Let $G$ be a graph and $S \subset V(G)$. A vertex $v \in V(G\backslash S)$ is said to be a \emph{universal vertex} if $\forall ~ x \in S, \{x, v\} \in E(G)$. An edge $\{u, v\}$ is said to be a \emph{universal edge} if $\forall ~x \in S$, either $\{x, u\} \in E(G)$ or $\{x, v\} \in E(G)$.
\end{dfn}

\begin{theorem}
Let $G$ be a connected graph and $S$ be any minimal vertex separator of $G$. Let $G_1, G_2, \ldots , G_l$, $(l \geq 2)$ be the connected components in $G\backslash S$. $G$ is $2K_2$-free if and only if it satisfies the following conditions:
\begin{itemize}
\item[(i)] $G\backslash S$ contains at most one non-trivial component. Further, if $G\backslash S$ has a non-trivial component, say $G_1$, then the graph induced on $V(G_1) \cup V(S)$ does not contain $H_1$, $H_2$, $H_3$ as an induced subgraphs.
\item[(ii)] Every trivial component of $G\backslash S$ is universal to $S$.
\item[(iii)] Every edge in the non-trivial component of $G\backslash S$ is universal to $S$.
\item[(iv)] The graph induced on $V(S)$ is either connected or has at most one non-trivial component. Further, if the graph induced on $V(S)$ has a non-trivial component, say $S_1$, then the graph induced on $V(S_1)$ does not contain $H_1$, $H_2$, $H_3$ as an induced subgraphs.
\item[(v)] If $S$ and $G\backslash S$ has a non-trivial component, say $S_1$ and $G_1$, respectively, then every edge $e=\{u,v\}$ in $S_1$ is universal to $M$, where $M = \{x \in G_1 \mid  \{x,y\} \in E(G), ~\forall ~y \in (S\backslash \{N_G(u)\cup N_G(v)\})\}$.
\end{itemize}
\end{theorem}
\begin{proof}
\textbf{Necessity:} Let $G$ be a $2K_2$-free graph.
\begin{description}

\item[(i)] On the contrary, assume that there are at least two non-trivial components in $G\backslash S$. Without loss of generality, let $G_1$ and $G_2$ be any two non-trivial components in $G\backslash S$, then there exists at least one edge $e = \{u, v\}$ and $f = \{x, y\}$ in $G_1$ and $G_2$, respectively. The edges $e$ and $f$ are not adjacent and forms $2K_2$ in $G$, which is a contradiction. Hence, our assumption that there are at least two non-trivial components in $G\backslash S$ is wrong.
Further, since every induced subgraph of a $2K_2$-free graph is $2K_2$-free, the graph induced on $V(G_1)\cup V(S)$ is $2K_2$-free. By \emph{Lemma 1}, $V(G_1)$ does not contain an induced $H_1$, $H_2$, $H_3$.

\item[(ii)] On the contrary, assume that there exist a trivial component, say $G_1$, in $G\backslash S$ such that the vertex $v \in V(G_1)$ is not adjacent to at least one vertex, say $p$, in $S$. The set $S\backslash \{p\}$ forms a vertex separator for $G$, which is a contradiction to our assumption $S$ is a minimal vertex separator. 
\item[(iii)] If $G\backslash S$ has only trivial components, then there is nothing to prove. Assume that there exists a non-trivial component in $G\backslash S$, say $G_1$. On the contrary, assume that there exist an edge $e = \{u, v\} \in E(G_1)$, which is not universal to $S$. i.e., there exists at least one vertex, say $p \in V(S)$, such that both $\{u, p\} \notin E(G)$ and $\{v, p\} \notin E(G)$. Since $S$ is a minimal vertex separator, there exist a vertex $w \in V(G_i)$, $2 \leq i \leq l$ such that $\{w, p\} \in E(G)$. This implies, the edges $\{u, v\}$ and $\{w, p\}$ forms a $2K_2$ in $G$, which is a contradiction. Hence, every edge in non-trivial component of $G\backslash S$ is universal to $S$.
\item[(iv)] On the contrary, the graph induced on $V(S)$ has at least two non-trivial components, say $S_1, S_2, \ldots, S_k$, $k \geq 2$. Then there exist edges $e \in E(S_1)$ and $f \in E(S_2)$ which forms $2K_2$ in $G$, which is a contradiction.
Thus, the graph induced on $V(S)$ has at most one non-trivial component. Since, every induced subgraph of a $2K_2$-free graph is $2K_2$-free, the graph induced on $V(S_1)$ is $2K_2$-free. By \emph{Lemma 1}, $V(S_1)$ does not contain an induced $H_1$, $H_2$, $H_3$.
\item[(v)] On the contrary, assume that there exist an edge $e = \{u, v\} \in E(S_1)$ is not universal to $M$. i.e., there exist a vertex $x \in  V(G_1)$ and $y \in V(S\backslash (N_G(u)\cup N_G(v)))$ such that $\{x, y\} \in E(G)$ and $\{x, u\}, \{x, v\} \notin E(G)$. Thus,
$\{x, y\}$ and $\{u, v\}$ form $2K_2$ in $G$, which is a contradiction.
\end{description}
\textbf{Sufficiency:} Our claim is to prove every pair of edges in $G$ do not form $2K_2$. Let $e = \{u, v\}$ and $f =\{x, y\}$ be any two edges in $G$. Since, $S$ is a minimal vertex separator for $G$, $G$ is a graph induced on $V (S) \cup V (G_1) \cup V(G_2) \cup \ldots \cup V(G_l)$. Let $S_1, S_2, \ldots, S_k, k \geq 1$ be the connected components in $S$.
\begin{description}
\item[Case 1:]$e, f \in E([S])$\\
Then there exist a non-trivial component in $S$, say $S_1$. Thus, $e, f \in E(S_1)$. By (iv), $e$ and $f$ do not form $2K_2$.

\item[Case 2:]If $e \in E([S])$, $x \in V(S)$ and $y \in V(G_i)$, $1 \leq i \leq l$.
\begin{itemize}
\item[$\bullet$] If $G\backslash S$ has only trivial components. By (ii), $\{u, y\}, \{v, y\} \in E(G)$. Thus, $e$ and $f$ do not form $2K_2$.
\item[$\bullet$] If $G\backslash S$ has a non-trivial component. By (i) there exist exactly one non-trivial component, say $G_1$. If $y \in V(G_i)$, $2 \leq i \leq l$, then there is nothing to prove. If $y \in V(G_1)$, then $y \in V(G_1\backslash M)$. If $x \in (N_G(u)\cup N_G(v))$, then $e$ and $f$ do not form $2K_2$. If $x \notin (N_G(u)\cup N_G(v))$, then, $\{y, u\} \in E(G)$ or $\{y, v\} \in E(G)$ (By (v)), thus, $e$ and $f$ do not form $2K_2$.
\end{itemize}

\item[Case 3:] If $u, x \in V(S)$, $v \in V(G_i)$ and $y \in V(G_j)$, $i \neq j$.
By (i), $G_i$ or $G_j$ is a trivial component. Without loss of generality, assume that $G_i$ is a trivial component. By (ii), $\{v, x\} \in E(G)$, thus, $e$ and $f$ do not form $2K_2$.

\item[Case 4:] If $G_1$ is a non-trivial component in $G\backslash S$ and $e, f \in E(G_1)$, then by (i), $e$ and $f$ do not form $2K_2$.

\item[Case 5:] If $G_1$ is a non-trivial component in $G\backslash S$, $e \in E(G_1)$ and $f \in E(S)$, then by (iii), $e$ and $f$ do not form $2K_2$.

\item[Case 6:] If $G_1$ is a non-trivial component in $G\backslash S$, $v, y \in E(G_1)$ and $u, x \in E(S)$, then by (i), $e$ and $f$ do not form $2K_2$.
\end{description}
In all the above cases $e$ and $f$ do not form $2K_2$. Thus, $G$ is $2K_2$-free. Hence the theorem. $\hfill \qed$
\end{proof}


\subsection{Enumeration of all minimal vertex separators}

In any $2K_2$-free graph $G$, for every minimal vertex separator $S$, the graph $G\backslash S$ has atleast one trivial component and every trivial component is universal to the respective minimal vertex separator. Using this observation, we can enumerate all minimal vertex separators for a given $2K_2$-free graph. Now, we present the algorithm as follows:

\begin{algorithm}[h]
\caption{\tt Enumeration of all MVS ($2K_2$-free graph $G$)}
\begin{algorithmic}[1]
\STATE{\textbf{Input:} $2K_2$-free graph, $G$}
\STATE{\textbf{Output:} All minimal vertex separators of $G$}
\STATE{Let $V(G) = \{v_1, v_2, \ldots, v_n\}$}
\STATE{Let $flag = 0$}
\FOR{$i=1$ to $n$}
	\STATE{Let $S_i = N_G(v_i)$}
	\STATE{Let $G_1, G_2, \ldots, G_l$ be the connected components in $G\backslash S_i$}
		\FOR{$j=1$ to $l$}
			\IF{$G_j$ is trivial}
				\STATE{Check $G_j$ is universal to $S_i$ or not. If not, $flag = flag +1$}			
			\ENDIF
		\ENDFOR
	\IF{$flag = 0$}
		\STATE{Print $S_i$ forms a minimal vertex separator}
	\ENDIF
\ENDFOR
\end{algorithmic}
\end{algorithm}

\begin{lemma}
The algorithm $\mathtt{Enumeration ~of ~all~ MVS ()}$ enumerates all minimal vertex separator of a $2K_2$-free graph, $G$.
\end{lemma}
\begin{proof}
Every minimal vertex separator in $G$ has at most one non-trivial component. Since, $G$ is $2K_2$-free, a non-trivial component of any vertex separator, $S'$, is universal to $S'$. Moreover, every minimal vertex separator, $S$, in $G$ has at least one trivial component and all such trivial components are universal to $S$. If there exist a trivial component which is not universal to $S$ then, $S$ is not minimal. Using these observations, we list all minimal vertex separators of a $2K_2$-free graphs in \emph{Steps 5-15} of \emph{Algorithm 4}. Thus, algorithm $\mathtt{Enumeration~ of~ all~ MVS ()}$ outputs all possible minimal vertex separators. $\hfill \qed$
\end{proof}

\noindent  The time complexity of the algorithm $\mathtt{Enumeration ~of ~all~ MVS ()}$ is $O(n^2\Delta)$, which is polynomial in the input size, where $\Delta$ is the maximum degree for $G$.

\subsection{Constrained vertex separators}


It is known from \cite{dragon} that stable separator is NP-complete. The only graph class known in the literature where stable separator is polynomial-time solvable is the class of triangle free graphs. In this section, we show that minimum stable separator in $2K_2$-free graphs is polynomial-time solvable. As far as clique separator is concerned, from \cite{tarjan}, it is known that clique separator is polynomial-time solvable in general graphs. For the sake of completeness we also present an algorithm to find the minimum clique separator in $2K_2$-free graphs.

\noindent \textbf{Finding minimum clique and stable separator:}\\
We shall now describe an algorithm for finding minimum clique separator and minimum stable separator in $2K_2$-free graphs. We enumerate all minimal vertex separator for a $2K_2$-free graph, $G$, using the algorithm $\mathtt{Enumeration ~of ~all ~MVS ()}$. For every minimal vertex separator we check whether it is clique or stable, if so, we find the minimum of all such separators. This algorithm is correct because, the addition of a vertex to a separator which is neither clique nor independent, does not create a clique(stable) vertex separator. Note that, it is not necessary that every $2K_2$-free graph has a clique as well as stable separator. The time complexity of this algorithm is $O(n^2\Delta)+O(n\Delta ^2) = O(n^2\Delta)$, which is polynomial in the input size, where $\Delta$ is the maximum degree for $G$.

\noindent \textbf{Finding minimum connected vertex separator:}\\ 
 From \cite{sadagopan}, it is known that minimum connected vertex separator is NP-complete in chordality-5 graphs. In this paper, we show that on $2K_2$-free graphs, minimum connected vertex separator is polynomial-time solvable which is a non-trivial subclass of chordality-5 graphs.
 
 \noindent \textbf{Sketch of the algorithm:}
For the input $2K_2$-free graph $G$ with $n$ vertices, enumerate all minimal vertex separators. Now, we group the minimal vertex separators into two lists $L_1$ and $L_2$ such that $L_1$ consists of all minimal vertex separators, whose removal from $G$ gives exactly two components and $L_2$ consists of remaining minimal vertex separators. Next, we sort $L_1$ and $L_2$ based on its cardinality. Now, search for the minimum connected vertex separator in $L_1$ and $L_2$, if exists say $C_1$ and $C_2$ respectively. Finally, we compare $C_1$ (if exists), $C_2$ (if exists) and the first minimal vertex separator in $L_2$ along with one of its trivial components. Subsequently, the one with least cardinality is returned as output.

\begin{algorithm}[H]
\caption{\tt Minimum Connected vertex separator ($2K_2$-free graph $G$)}
\begin{algorithmic}[1]
\STATE{\textbf{Input:} $2K_2$-free graph, $G$}
\STATE{\textbf{Output:} Minimum connected vertex separator of $G$}
\STATE{\tt Enumeration of all MVS ($G$)}
\STATE{Let $S_1, S_2,\ldots, S_r$, ($r < n$) be the all possible minimal vertex separator of $G$}
\STATE{Let $C_1=n, C_2=n, j=0$ and $k=0$}
\STATE{Create two lists $L_1$ and $L_2$}
\FOR{$i=1$ to $r$}
	\IF{$c(G\backslash S_i) = 2$}
		\STATE{/* \tt $c(G\backslash S_i)$ denotes the number of connected components in $G\backslash S_i$ */}
		\STATE{Append $S_i$ to $L_1$}
		\STATE{$j=j+1$}
	\ELSE
		\STATE{Append $S_i$ to $L_2$}
		\STATE{$k=k+1$}
	\ENDIF
\ENDFOR
\STATE{Give an ordering to the set of separators in $L_1 = (a_1,a_2,\ldots, a_j)$ such that $\vert V(a_1) \vert < \vert V(a_2) \vert < \ldots < \vert V(a_j) \vert$}
\STATE{Give an ordering to the set of separators in $L_2 = (b_1,b_2,\ldots, b_k)$ such that $\vert V(b_1) \vert < \vert V(b_2) \vert < \ldots < \vert V(b_k) \vert$}
\FOR{$i=1$ to $j$}
	\IF{$a_i$ is connected}
		\STATE{$C_1 = \vert V(a_i)\vert$}
		\STATE{$p=i$}
		\STATE{Break the for loop}
	\ENDIF
\ENDFOR

\FOR{$i=1$ to $k$}
	\IF{$b_i$ is connected}
		\STATE{$C_2 = \vert V(b_i)\vert$}
		\STATE{$q=i$}
		\STATE{Break the for loop}
	\ENDIF
\ENDFOR
%
\IF{$C_1 = n$ and $C_2 \neq n$}
	\IF{$C_2 < \vert V(b_1) \vert+1$}
		\STATE{Return $V(b_q)$}
	\ELSE
		\STATE{Return $V(b_1)\cup \{u\}$, where $u$ is a vertex in any trivial component of 				   $G\backslash b_1$}
	\ENDIF
\ELSIF{$C_1 \neq n$ and $C_2 = n$}
	\IF{$C_1 < \vert V(b_1) \vert+1$}
		\STATE{Return $V(a_p)$}
	\ELSE
		\STATE{Return $V(b_1)\cup \{u\}$, where $u$ is a vertex in any trivial component of 				   $G\backslash b_1$}
	\ENDIF
\ELSIF{$C_1 \neq n$ and $C_2 \neq n$}
	\IF{$C_1 < C_2$ and $C_1 < \vert V(b_1) \vert+1$}
		\STATE{Return $V(a_p)$}
	\ELSIF{$C_2 < \vert V(b_1) \vert+1$}
		\STATE{Return $V(b_q)$}
	\ELSE
		\STATE{Return $V(b_1)\cup \{u\}$, where $u$ is a vertex in any trivial component of 				   $G\backslash b_1$}
	\ENDIF
\ENDIF
\end{algorithmic}
\end{algorithm}

\begin{lemma}
The algorithm $\mathtt{Minimum~ Connected ~vertex ~separator~ ()}$ returns minimum connected vertex separator.
\end{lemma}
\begin{proof}
We first enumerate all minimal vertex separators and then we partition the minimal vertex separators into two lists, namely $L_1$ and $L_2$, such that $L_1$ contains all the minimal vertex separators which gives exactly two components on its removal from the graph and $L_2$ contains all the minimal vertex separators which gives more than two components on its removal from the graph. Now we sort $L_1$ and $L_2$ based on the cardinality of the separators such that $L_1 = (a_1,a_2,\ldots, a_i)$, where $\vert V(a_1) \vert \leq \vert V(a_2) \vert \leq \ldots \leq \vert V(a_i) \vert$ and $L_2 = (b_1,b_2,\ldots, b_j)$, where $\vert V(b_1) \vert \leq \vert V(b_2) \vert \leq \ldots \leq \vert V(b_j) \vert$. Note that minimum of all the minimal vertex separators gives the minimum vertex separator, using this fact we search for the first connected vertex separators in $L_1$ and $L_2$ separately, say $a_p$ and $b_q$. If such $a_p$ and $b_q$ exist in $L_1$ and $L_2$, respectively, we compare $\vert V(a_p) \vert$, $\vert V(b_q) \vert$ and $\vert V(b_1) \vert+1$ and outputs the minimum, where $G\backslash V(b_1)$ has more than two components so $V(b_1)$ along with the trivial component forms a connected vertex separator. If $a_p$ does not exist, then compare $\vert V(b_q) \vert$ and $\vert V(b_1) \vert+1$ and outputs the minimum. If $b_q$ does not exist, then compare $\vert V(a_p) \vert$ and $\vert V(b_1) \vert+1$ and outputs the minimum. If both does not exist $\vert V(b_1) \vert+1$ will be the minimum. Thus, the algorithm $\mathtt{Minimum~ Connected ~vertex ~separator~ ()}$ outputs the minimum connected vertex separator.
$\hfill \qed$
\end{proof}

\noindent  The algorithm $\mathtt{Minimum~ Connected ~vertex ~separator~ ()}$ takes $O(n^2\Delta)$ in \emph{Step 3}, $O(n(n+m)+ \Delta)$ for \emph{Steps 7-16}, $O(n ~log n)$ for \emph{Steps 17-18}, $O(n+m)$ for \emph{Steps 19-32} and constant time for the \emph{Steps 33-53}. Thus, the time complexity of the algorithm $\mathtt{Minimum~ Connected ~vertex ~separator~ ()}$ is $O(nm)$, which is polynomial in the input size. 

\section{Enumeration of all maximal independent sets}
 Farber proves that the complement class of $2K_2$-free graphs ($C_4$-free graphs) has polynomial number of maximal cliques \cite{farber}. From this we can say that there are polynomial number of maximal independent sets in $2K_2$-free graphs. We enumerate all possible maximal independent sets for a given $2K_2$-free graph in the following algorithm.
 
\begin{algorithm}[H]
\caption{\tt Enumeration$\_$of$\_$all$\_$MIS ($2K_2$-free graph $G$)}
\begin{algorithmic}[1]
\STATE{\textbf{Input:} $2K_2$-free graph, $G$}
\STATE{\textbf{Output:} All maximal independent sets of $G$}
\STATE{Let $S_1, S_2,\ldots, S_r$, ($r < n$) be the all possible minimal vertex separator of $G$}
\FOR{$i=1$ to $r$}
	\STATE{$I = \emptyset$}
	\STATE{$I = I \cup \{$trivial components in $G\backslash S_i\}$}
	\IF{$G\backslash S_i$ has a non-trivial component, say $G_1$}	
					\STATE{$I = I \cup $ MIS($G_1$)}
					\STATE{Print $I$}
	\ELSE
		\STATE{Print $I$}	
		\STATE{MIS($G\backslash I$)}		
	\ENDIF
\ENDFOR
\end{algorithmic}
\end{algorithm}

\newpage 

\begin{lemma}
Let $G$ be a $2K_2$-free graph. The set $I$, output by algorithm $\mathtt{Enumeration\_of\_all\_MIS( )}$, is a maximal independent set.
\end{lemma}
\begin{proof}
First we have to prove that the algorithm outputs only the independent sets. i.e., to prove for any two vertices in $I$, say $u, v$,  there does not exists an edge $\{u,v\}\in E(G)$. Consider a minimal vertex separator $S$ in $G$, let $G_1, G_2, \ldots, G_r$ be the components in $G\backslash S$. If $G\backslash S$ has only trivial components then the algorithm returns the trivial components as one of the solution, which is clearly a maximal independent set and the remaining solutions are obtained by considering $S$. If $G\backslash S$ has a non-trivial component, say $C_1$, then the set $I$ consists of elements of $C_2, \ldots, C_r$ (which are trivial components) and the other elements of $I$ are obtained recursively by considering $C_1$. The algorithm terminates when the non-trivial component is a clique. Let the solution returned by $C_1$ is $I'$. Since $C_1$ is a connected component of $G\backslash S$, clearly any element of $I'$ does not have any adjacency from $C_2,\ldots, C_r$. Thus, the set output by the algorithm is a independent set. Now, our claim is to prove that every set is maximal. Let $I$ be any independent set, which results from the algorithm. Every vertex in $G\backslash I$ is a vertex in some separator, thus an addition of a vertex from $G\backslash I$ to $I$ will contradict the definition of independent set, by \emph{Theorem 1.(ii)}. Therefore, $I$ is a maximal independent set. Hence the Lemma. $\hfill \qed$ 

\end{proof}

\begin{lemma}
The  algorithm $\mathtt{Enumeration\_of\_all\_MIS( )}$ enumerates all maximal independent sets for the given connected $2K_2$-free graph, $G$.
\end{lemma}
\begin{proof}
We shall prove the lemma by mathematical induction on the number of vertices $n$ of $G$.\\
\textit{Basis Step:} For all the connected $2K_2$-free graphs on $n$ vertices, $1\leq n \leq 3$, the algorithm enumerates all the maximal independent sets. \\
\textit{Hypothesis:} Assume that the Lemma is true for connected $2K_2$-free graphs with fewer vertices than $n$, $n \geq 4$\\
\textit{Induction Step:} Let $G$ be a connected $2K_2$-free graph with $n$ vertices. Our claim is to prove that the algorithm enumerates all the maximal independent sets for $G$. Let $S=N_G(v)$, where $v$ is the minimum degree vertex of $G$. If $G\backslash S$ has only trivial components, then $G\backslash S$ forms one maximal independent set and by the hypothesis, all other maximal independent sets of $S$ will be enumerated. If $G\backslash S$ has a non-trivial component, then by hypothesis, the algorithm enumerates all the maximal independent sets in the non-trivial component of $G\backslash S$, say $I_1, I_2, \ldots, I_k$. Now, add the trivial components in $G\backslash S$ to each of the maximal independent set $I_i$, $1\leq i \leq k$, which is a maximal independent set for $G$. Thus, all the maximal independent sets of $G$ are enumerated. $\hfill \qed$ \\
\end{proof}

\begin{corollary}
Let $G$ be a $2K_2$-free graph. Enumeration of all minimal vertex cover of $G$ can be done in polynomial-time and thus, the minimum vertex cover of $G$.
\end{corollary}
\begin{proof}
Enumeration of all minimal vertex cover can be obtained by taking the complement of each maximal independent set obtained in the \emph{Algorithm 7}. Thus, the minimum vertex cover for a given connected $2K_2$-free graph can be obtained by searching for the minimal vertex cover with the least cardinality. $\hfill \qed$\\
\end{proof}

\noindent The algorithm $\mathtt{Enumeration\_of\_all\_MIS( )}$ enumerates all maximal independent sets for a given connected $2K_2$-free graphs. Thus, the maximum independent set for a given connected $2K_2$-free graph can be obtained by searching for the maximal independent set with the largest cardinality. In algorithm $\mathtt{Enumeration\_of\_all\_}$ $\mathtt{MIS(G)}$: The time complexity for enumerating all minimal vertex separators for $G$ in \emph{Step 3} is $O(n^2\Delta)$, the \emph{Steps 4-14} takes $O(n^2m)$ time. Thus, the algorithm takes $O(n^2m)$ time, which is polynomial in the input size.

\noindent A well-known combinatorial problem, $3$-Colorability in $2K_2$-free graphs is known to be solvable in polynomial time from its super graph class $P_5$-free graphs \cite{shu}. We provide an alternative algorithm for 3-colorability in terms of MVS and MIS. 

\begin{lemma}
\label{color}
For any connected 3-colorable graph $G$, at least one of the proper coloring has an MIS as one of its color class.
\end{lemma}

\begin{proof}
Since, $G$ is $3$-colorable, it has three color classes, namely $C^1, C^2, C^3$. Since, $C^i, i =1,2,3$ is a proper coloring for $G$, each $C^i$  is an independent set. If any one of $C^i, i =1,2,3$ forms an MIS for $G$, then the lemma is true. Assume that none of the color classes, $C^i, i =1,2,3$, forms an MIS for $G$, then consider the color class $C^1$, add vertices to $C^1$ from $C^2$ and $C^3$ such that it obeys the property of independent set. Do this process, until no vertices can be added to $C^1$ from $C^2$ and $C^3$. Thus, $C^1$ forms an MIS for $G$. Hence proved. $\hfill \qed$
\end{proof}

\noindent We now describe an algorithm for checking whether the given $2K_2$-free graph, $G$, is $3$-colorable or not. First, we check whether $G$ is bipartite or not. If so, we say $G$ is 2-colorable. If not, we enumerate all maximal independent sets and for every maximal independent set, $I$, we check whether $G\backslash I$ is bipartite or not. If so, we say $G$ is 3-colorable. If not, $G$ is not 3-colorable. The algorithm is correct, by \emph{Lemma \ref{color}}. This algorithm is applicable for any graph class in which enumerating all maximal independent set can be done in polynomial-time.

\section{Testing $2K_2$-free graphs}
As mentioned before, if $G$ satisfies all the conditions in \emph{Theorem 1}, then $G$ is a $2K_2$-free graph. We next show an interesting claim which says that in any $2K_2$-free graphs, neighborhood of a minimum degree vertex is a minimal vertex separator. This claim helps us to choose one minimal vertex separator from the given graph and the testing can be done more efficiently. 

\begin{theorem}
Let $G$ be a connected $2K_2$-free graph and $u \in V(G)$, $d_G(u) = \delta (G)$. Then,  $N_G(u)$ is a minimal vertex separator of $G$.
\end{theorem}
\begin{proof}
Let $u$ be a minimum degree vertex in $G$ such that $d_G(u) = k$ and $d_G(v) \geq k$, $\forall~ v \in V(G)\backslash \{u\}$. Let $S = N_G(u)$ be a vertex separator of size $k$. Let $G_1, G_2, \ldots, G_l$ $(l \geq 2)$ be the connected components of $G\backslash S$ and let $V(G_1) = \{u\}$. If $d_G(v) = k$, $\forall ~ v \in V(G\backslash S)$, then there is nothing to prove. Assume that there exist at least one vertex, say $w$, in $G\backslash S$ such that $d_G(w) > k$. Therefore, there exist an edge $\{w, x\}$ such that $w, x \in V (G\backslash S)$ and belongs to one connected component in $G\backslash S$, say $G_2$. Our claim is to prove that $S$ is a minimal vertex separator. On the contrary, assume that $S$ is not minimal, then there exists a proper subset $S'$ of $S$ such that $S'$ forms a minimal vertex separator. This implies, in $G\backslash S'$, $(S\backslash S') \cup \{u\}$ forms a non-trivial connected component. $G_2$ is also a non-trivial connected component in $G\backslash S'$, which is a contradiction as per \emph{Theorem 2}. Hence our assumption that $S$ is not a minimal vertex separator is wrong. Therefore, $N_G(u)$ is a minimal vertex separator. $\hfill \qed$
\end{proof}

\noindent Using all the above mentioned observations, we can test whether a given graph is $2K_2$-free or not. Now, we present an algorithm for testing whether the given graph is $2K_2$-free or not.\\

\begin{algorithm}[h]
\label{test}
\caption{\tt Test$\_2K_2$ free (Graph $G$)}
\begin{algorithmic}[1]
\STATE{\textbf{Input:} Graph $G$}
\STATE{\textbf{Output:} $G$ is $2K_2$-free or not}
\STATE{Let $u$ be a minimum degree vertex in $G$}
\STATE{Let $S = N_G(u)$ and $j = 0$}
\STATE{Let $G_1, G_2, \ldots, G_l$ be the connected components of $G\backslash S$ and let $G_1 = \{u\}$.}
\IF{$G\backslash S$ has more than one non-trivial component}
	\STATE{Return $G$ is not $2K_2$-free }
\ELSE
	\FOR{$i=2$ to $l$}
		\IF{$G_i$ is trivial}
			\STATE{Check $G_i$ is universal to $S$ or not. If not, return $G$ is not $2K_2$-free}			
		\ELSE
			\STATE{Initialize $j=i$ and Check every edge in $G_i$ is universal to $S$ or not. If not, return $G$ is not $2K_2$-free}
		\ENDIF
	\ENDFOR
\ENDIF
\STATE{Let $S_1, S_2, \ldots, S_t$ be the components of $S$}
\IF{$S$ has more than one non-trivial component}
	\STATE{Return $G$ is not $2K_2$-free}
\ELSIF{$S$ has a non-trivial component, say $S_{nt}$, and $j \neq 0$}
		
				\STATE{Check every edge, $\{u,v\}$ in $S_{nt}$ is universal to $M$ or not, where $M = \{x \in G_j \mid  \{x,y\} \in E(G), ~\forall ~y \in (S\backslash \{N_G(u)\cup N_G(v)\})\}$. If not, return $G$ is not $2K_2$-free}	 
\ENDIF
		 	\IF{$\vert V(S_{nt}) \vert \geq 5$ and $j =0$}
		 		\STATE{\tt Test$\_2K_2$ free($S_i$)}
		 	\ELSIF{$j \neq 0$}
		 		\STATE{\tt Test$\_2K_2$ free($G_1 \cup S$)}
		 	\ENDIF

\STATE{Return $G$ is $2K_2$-free graph}
\end{algorithmic}
\end{algorithm}

\noindent The algorithm $\mathtt{Test}\_2K_2$ $\mathtt{free}(G)$ checks whether the given graph satisfies all the conditions of \emph{Theorem 1}: the first three conditions are checked inside the loop in \emph{Steps 6-16, 27}, where, condition (i) in \emph{Steps 6-7} and in \emph{Step 27}, condition (ii) in \emph{Steps 10-11}, condition (iii) in \emph{Step 13} and the conditions (iv) and (v) are checked in \emph{Steps 18-23} and in \emph{Step 25}. Since, these conditions will obviously check the minimality of $S$ (by condition (ii) and (iii)), the algorithm $\mathtt{Test}\_2K_2$ $\mathtt{free}(G)$ outputs whether the given graph is $2K_2$-free or not. \\

\noindent Let $G$ be a graph with $n$ vertices and $m$ edges. In algorithm $\mathtt{Test}\_2K_2$ $\mathtt{free}(G)$:  The time complexity for checking vertex universal is $O(\delta)$, where $\delta$ is the number of vertices in $S$. The number of edges in $S$ is at most $\delta^2$ and it takes $O(n\Delta)$ times for checking the edge universal from the non-trivial component to $S$. Hence, for \emph{Steps 6-16} the algorithm takes $O(n\Delta)$ times. It takes $O(\delta^2\Delta)$ times for checking the edge universal from the $S$ to non-trivial component. Hence, for \emph{Steps 18-23} the algorithm takes $O(\delta^2\Delta)$. Overall, the algorithm $\mathtt{Test}\_2K_2$ $\mathtt{free}(G)$ takes $O(n(\delta^2\Delta+n\Delta))$, which is polynomial in the input size, where $\Delta$ is the maximum degree for $G$ and $\delta$ is the minimum degree of $G$. \\

\noindent Time complexity of existing algorithms for testing $2K_2$-free on $n$ vertices: Trivial algorithm takes $O(n^4)$ time for recognizing $2K_2$-free graphs. We also highlight that, $C_4$-free testing can be done in $O(n^{3.3333953})$ time \cite{david}. Note that, testing $2K_2$-free is equivalent to testing $C_4$-free in the complement graph. Our algorithm takes $O(n(\delta^2\Delta+n\Delta))$ time. Further, for $\Delta$-bounded graphs, our algorithm runs in $O(n^2)$ time. Although, in general our algorithm is as good as trivial algorithm, it is efficient for $\Delta$-bounded graphs. It is important to highlight that $\Delta$-bounded graphs are extensively studied in the literature in the study of $(\Delta+1)$-coloring in linear (in $\Delta$) time \cite{leo}, 3-colorability of small diameter graphs \cite{george}, providing bounds for domatic and chromatic number \cite{andreas} etc., 

\subsection{Trace of Algorithm \ref{test}}
In this section, we trace \emph{Algorithm \ref{test}} with two illustrations using which we highlight all cases of \emph{Theorem 1}. \\

\noindent
\textbf{Illustration 1:} We trace the steps of algorithm $\mathtt{Test}\_2K_2$ $\mathtt{free}(G)$ in the Fig. 2.

\begin{itemize}
\item[1.] Consider the graph $G$ with nine vertices as illustrated in \emph{Fig. 2}. The vertex labelled $1$ is a minimum degree vertex with $S = N_G(1) = \{2,3,7\}$. 
\item[2.] As shown in \emph{Fig. 2.}, $G_1$ and $G_2$ are the connected components in $G\backslash S$, where $G_2$ is a non-trivial component. 

\item[3.] When we call the algorithm $\mathtt{Edge\_Universal(G,G_2,S)}$, it returns $1$, as all the edges in $G_2$ are universal to $S$.

\item[4.] Since $\vert V(G_2) \vert \geq 5$, recursively call the algorithm $\mathtt{Test}\_2K_2$ $\mathtt{free}(G_2)$.

\item[5.] Now, the minimum degree vertex in $G=G_2$ is $9$ and the vertex separator $S=\{8\}$. $G_1$ and $G_2$ are the connected components in $G\backslash S$, where $G_2$ is a non-trivial component. 

\item[6.] When we call the algorithm $\mathtt{Edge\_Universal(G,G_2,S)}$, it returns $0$, as the edge $\{5,6\}$ is not universal to $S$. The edges $\{8,9\}$ and $\{5,6\}$ form $2K_2$ in $G$. Hence, the given graph $G$ is not $2K_2$-free. The algorithm is complete.
\end{itemize}

\begin{figure}
\begin{center}
\includegraphics[scale=0.25]{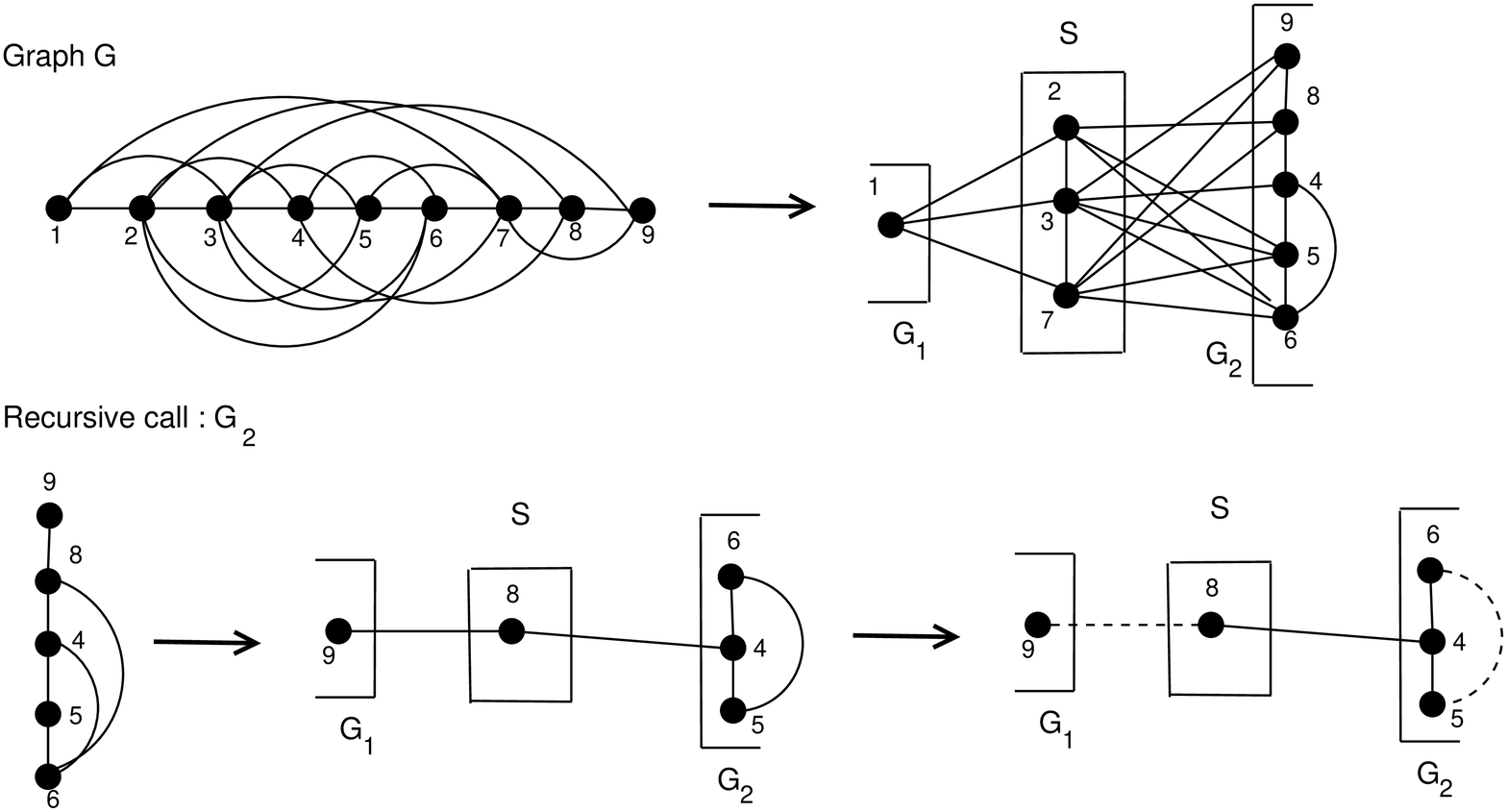}
\caption{An example for testing a graph is $2K_2$-free or not using Algorithm \ref{test}}
\end{center}
\end{figure}

\noindent
\textbf{Illustration 2:} We trace the steps of algorithm $\mathtt{Test}\_2K_2$ $\mathtt{free}(G)$ in the Fig. 3.

\begin{itemize}
\item[1.] Consider the graph $G$ with eight vertices as illustrated in \emph{Fig. 3}. The vertex labelled $1$ is a minimum degree vertex with $S = N_G(1) = \{2,3,4\}$.

\item[2.] As shown in \emph{Fig. 3.}, $G_1$, $G_2$, $G_3$ and $G_4$ are the connected components in $G\backslash S$, where $G_3$ is a non-trivial component. 

\item[3.] When we call the algorithm $\mathtt{Vertex\_Universal(G,S,V(G_2))}$, it returns $1$, as the vertex $5$ in $G_2$ is universal to $S$.

\item[4.] When we call the algorithm $\mathtt{Edge\_Universal(G,G_3,S)}$, it returns $1$, as all the edges in $G_3$ are universal to $S$. Note that, $\vert V(G_3) \vert <5$.

\item[5.] When we call the algorithm $\mathtt{Vertex\_Universal(G,S,V(G_4)}$, it returns $1$, as the vertex $8$ in $G_4$ is universal to $S$.

\item[6.] Since $S$ has only trivial components, the algorithm is complete and outputs that $G$ is a $2K_2$-free graph.

\end{itemize}

\begin{figure}
\begin{center}
\includegraphics[scale=0.25]{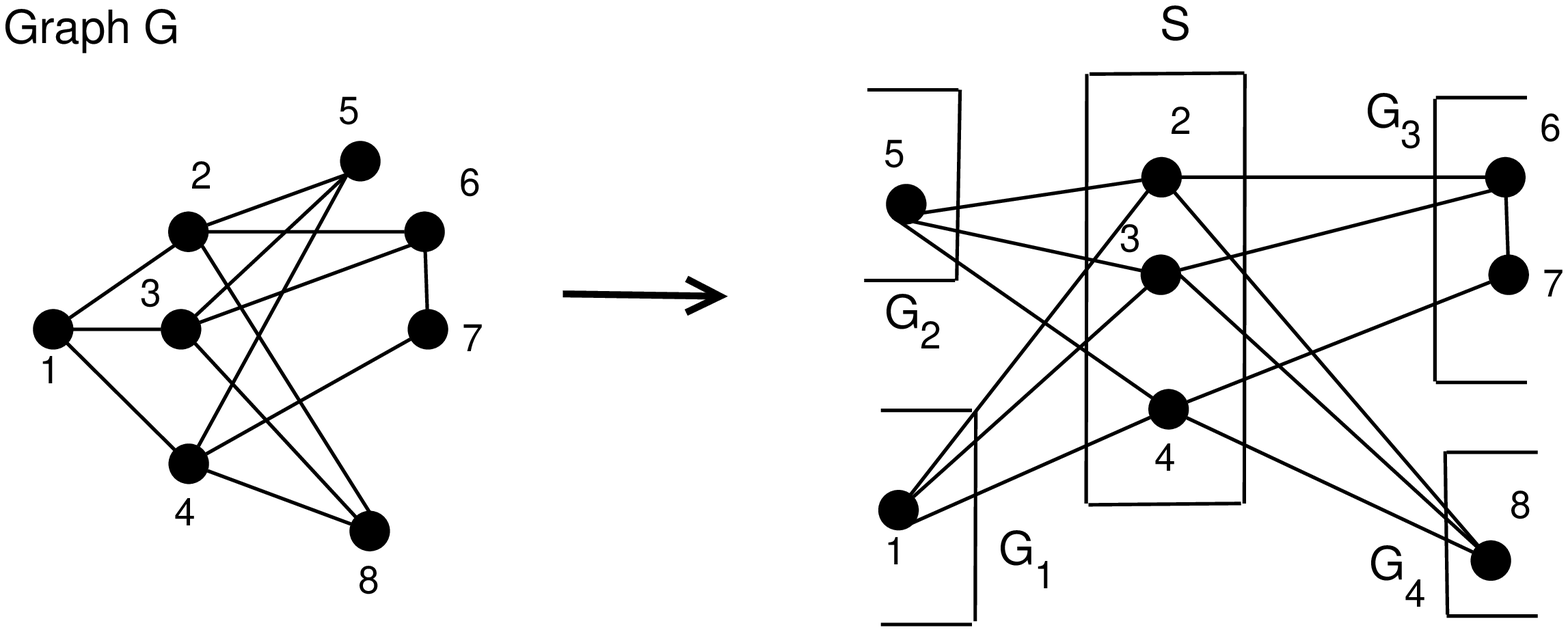}
\caption{An example for testing a graph is $2K_2$-free or not using Algorithm \ref{test}}
\end{center}
\end{figure}

\section{Feedback vertex set in subclass of $2K_2$-free graphs}

Note that by definition $2K_2$-free forbids $C_n, n \geq 6$ as an induced subgraph. However, $2K_2$-free graphs can allow $C_3$, $C_4$ and $C_5$. We study subclass of $2K_2$-free graphs by forbidding one or two cycles from $C_3$, $C_4$ and $C_5$. Further, we study the famous combinatorial problem which asks for the minimum number of vertex subset for the given graph whose removal creates a forest. In the literature, this problem is known to be feedback vertex set problem.

\subsection{Structural Results of $(2K_2, C_3, C_4)$-free graphs}
$(2K_2, C_3, C_4)$-free graphs form a proper subclass of $2K_2$-free graphs, where every induced cycle is of length 5. We observed the following structural properties and concluded that it is a small graph class.

\begin{theorem}
\label{c3c4}
If $G$ is a connected $(2K_2, C_3, C_4)$-free graph, then for any minimal vertex separator $S$ of $G$ satisfies the following properties:
\begin{itemize}
\item[(i)] $S$ is an independent set.
\item[(ii)] If $\mid S \mid > 1$, then $G\backslash S$ have exactly one trivial component.
\item[(iii)] If $G\backslash S$ has a non-trivial component, say $G_1$, then for every edge $\{u,v\} \in E(G_1)$, $(N_G(u) \cap V(S)) \cap (N_G(v) \cap V(S)) = \emptyset$ and $(N_G(u) \cap V(S)) \cup (N_G(v) \cap V(S)) = S$. i.e., For every vertex $x \in V(S)$, $(N_G(x) \cap V(G_1))$ is an independent set.
\item[(iv)] Every vertex in a non-trivial component is adjacent to exactly one vertex in $S$.
\end{itemize}
\end{theorem}
\begin{proof}
\begin{itemize}
\item[(i)] On the contrary, assume that $S$ has at least one edge, say $\{x,y\}$. Let $G_i$ be  a trivial component in $G\backslash S$ and let $V(G_i) = \{w\}$. Since, $G$ is a $2K_2$-free graph, $\{w,x\}, \{w,y\} \in E(G)$ (by \emph{Theorem 1.(ii)}). Thus, $(w,x,y)$ forms an induced $C_3$, which is a contradiction to the definition of $G$. Hence, $S$ is an independent set.
\item[(ii)] On the contrary, assume that $G\backslash S$ has at least two trivial components, say $G_i$ and $G_j$. Let $V(G_i) = \{w_i\}$ and $V(G_j) = \{w_j\}$. Let $x,y$ be any two vertices in $S$. By $(i)$, $\{x,y\} \notin E(G)$ and by \emph{Theorem 1.(ii)}, $\{w_i,x\}$, $w_i, y\}$, $\{w_j, x\}$, $\{w_j, y\} \in E(G)$. Thus, $(w_i,x,w_j,y)$ forms an induced $C_4$, which is a contradiction to the definition of $G$. Hence, $G\backslash S$ have exactly one trivial component if $\mid S \mid >1$.
\item[(iii)] By \emph{Theorem 1.(iii)}, every edge $\{u,v\} \in E(G_1)$ is universal to $S$, thus, $(N_G(u) \cap V(S)) \cup (N_G(v) \cap V(S)) = S$. Moreover, if $(N_G(u) \cap V(S)) \cap (N_G(v) \cap V(S)) \neq \emptyset$, then every vertex in $(N_G(u) \cap V(S)) \cap (N_G(v) \cap V(S))$ forms an induced $C_3$ together with $u$ and $v$. Hence, $(N_G(u) \cap V(S)) \cap (N_G(v) \cap V(S)) = \emptyset$.
\item[(iv)] On the contrary, assume that exist a vertex $v$ in a non-trivial component such that $(N_G(v) \cap V(S)) = \{x_1, x_2, \ldots, x_p\},$ $p \geq 2$. By $(ii)$, there exist a trivial component in $G \backslash S$, say $G_2$. Let $V(G_2) = \{w\}$. Therefore, $(v, x_1, x_2, w)$ forms an induced $C_4$, which is a contradiction to the definition of $G$. 
\end{itemize}
 Hence, the theorem.
$\hfill \qed$
\end{proof}

From the above structural observations of minimal vertex separators in $(2K_2, C_3, C_4)$-free graphs, $G$: we can observe that the only possible structure of a non-trivial component after the removal of any minimal vertex separator from $G$ is $K_2$ and $\mid S \mid \leq 2$. Further, if $\mid S \mid = 1$, then the graph is $(2K_2,cycle)$-free. If $\mid S \mid = 2$ and if $G\backslash S$ has a non-trivial component, then the graph is an induced $C_5$. Thus, feedback vertex set problem can be solved in $O(1)$ for $(2K_2, C_3, C_4)$-free graphs.

\subsection{Structural Results of $(2K_2, C_3, C_5)$-free graphs}
$(2K_2, C_3, C_5)$-free graphs are $2K_2$-free graphs where every induced cycle is of length 4. This graphs can also be called as $2K_2$-free chordal bipartite graphs. The structural observations for this graphs are as follows:
\begin{theorem}
\label{c3c5}
If $G$ is a connected $(2K_2, C_3, C_5)$-free graph, then for any minimal vertex separator $S$ of $G$ satisfies the following properties:
\begin{itemize}
\item[(i)] $S$ is an independent set.
\item[(ii)] If $G\backslash S$ has a non-trivial component, say $G_1$, then for every vertex $x \in V(S)$, $(N_G(x) \cap V(G_1))$ is an independent set.
\item[(iii)] For every edge $\{u,v\}$ in a non-trivial component $G_1$ of $G\backslash S$, either $u$ is universal to $S$ and $(N_G(v) \cap V(S)) = \emptyset$.
\end{itemize}
\end{theorem}
\begin{proof}
\begin{itemize}
\item[(i)] The argument is similar to the proof in \emph{Theorem \ref{c3c4}.(i)}.
\item[(ii)] The argument is similar to the proof in \emph{Theorem \ref{c3c4}.(iii)}.
\item[(iii)] On the contrary, assume that there exist an edge $\{u,v\} \in E(G_1)$ such that $(N_G(u) \cap V(S)) \neq S$, $(N_G(v) \cap V(S)) \neq \emptyset$ and $(N_G(u) \cap V(S)) \neq \emptyset$. Since, $G$ is $2K_2$-free graph, $(N_G(u) \cap V(S)) \cup (N_G(v) \cap V(S)) = S$ and there exist a trivial component in $G\backslash S$, say $G_2$. Let $V(G_2) = \{w\}$. By our assumption, $u$ is adjacent to some vertex in $S$, say $x$ and $v$ is adjacent to some vertex in $S$, say $y$, such that $x \neq y$. Thus, $(u,v, y, w, x)$ forms an induced $C_5$, which is a contradiction to the definition of $G$.

\end{itemize}
$\hfill \qed$
\end{proof}

Although, it is known that the problem of finding minimum feedback vertex set in chordal bipartite graphs, a super class of $2K_2$-free chordal bipartite graphs, is polynomial time solvable \cite{ton}, using the above observation we provide a different approach for this problem in $(2K_2, C_3, C_5)$-free graph. Moreover, our approach takes linear time in terms of the input size.

\begin{theorem}
\label{fvsc3c5}
Let $G$ be a connected $(2K_2, C_3, C_5)$-free graph and $S$ be any minimal vertex separator of $G$, then the cardinality of minimum feedback vertex set, $F$, is equal to 
\begin{itemize}
\item[(i)] $min \{\mid S \mid -1, \mid T \mid - 1\}$, if $G\backslash S$ has only trivial components, where $T$ is the set of all trivial components in $G\backslash S$.
\item[(ii)] $min \{ \mid S \mid, \mid U \mid + (\mid T \mid -1)\}$, if $G\backslash S$ has a non-trivial component, $G_1$, and $G_1$ is cycle-free, where $U$ is the set of all vertices in $G_1$ which are universal to $S$.
\item[(iii)] $min \{ \mid U \mid + (\mid  T \mid -1), (\mid U \mid -1) + (\mid S \mid -1)\}$, if $G\backslash S$ has a non-trivial component, $G_1$ and $G_1$ has at least one cycle.
\end{itemize}
\end{theorem}
\begin{proof}
\begin{itemize}
\item[(i)] If $G$ is a cycle-free graph, then either $\mid S \mid = 1$ or $\mid T \mid = 1$. Thus, $F = \emptyset$, which is minimum. Without loss of generality, assume that $G$ has at least one cycle and $G \backslash S$ has only trivial components, say $G_1, G_2, \ldots, G_l$, $l \geq 2$. By our assumption, $\mid S \mid \geq 2$ and by \emph{Theorem \ref{c3c5}}, $S$ is an independent set. Let $V(G_i)=\{u_i\}$. Clearly, $G \backslash F$ results in a forest, where $\mid S \mid$ - 1 vertices are chosen from $S$ and $\mid T \mid$ - 1 vertices are chosen from $T$. Now, our claim is to prove the set $F$ is minimum.
\begin{itemize}
\item[$\bullet$] $F = min \{\mid S \mid -1, \mid T \mid - 1\} = \mid S \mid - 1$ \\
On the contrary, assume that $F$ is not minimum, then the removal of $S'$ vertices from $G$ results in a forest, where $S' <  \mid S \mid - 1$. i.e., $S$ has at least two vertices in $G \backslash F$, say $x, y \in V(S)$. Therefore, $(u_1, x, u_2, y)$ forms an induced $C_4$, which is a contradiction to the definition of $F$. 
\item[$\bullet$]  $F = min \{\mid S \mid -1, \mid T \mid - 1\} = \mid T \mid - 1$ \\
On the contrary, assume that $F$ is not minimum, then the removal of $T'$ vertices from $G$ results in a forest, where $T' <  \mid T \mid - 1$. i.e., $T$ has at least two vertices in $G \backslash F$, say $u_1, u_2 \in T$. Let $x$ and $y$ be any two vertices in $S$. Therefore, $(u_1, x, u_2, y)$ forms an induced $C_4$, which is a contradiction to the definition of $F$. 
\end{itemize}
Hence, $F$ is the minimum feedback vertex set if $G\backslash S$ has only trivial components.
\item[(ii)] All possible structures of $G_1$ is given in \emph{Fig.4}. From the structures of $G_1$, it is clear that $F$ is a minimum feedback vertex set.
\begin{figure}[h]
\label{g123}
\centering
\includegraphics[scale=0.48]{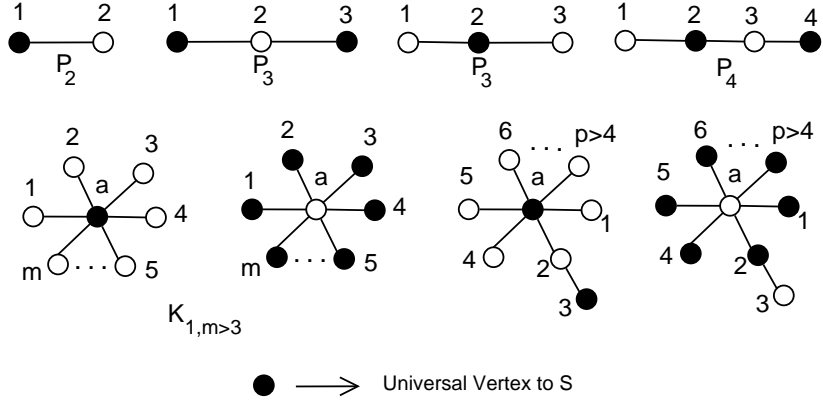}
\caption{All Possible structures of $G_1$ when $G_1$ is cycle-free}
\end{figure}
\item[(iii)] We prove this case separately for $\mid S \mid = 1$ and $\mid S \mid >1$.
\begin{itemize}
\item[$\bullet$] $\mid S \mid = 1$ and let $S = \{x\}$. \\
It is clear that, every cycle of $G$ lies in $G_1$. Thus, $F = min \{ \mid U \mid + (\mid  T \mid -1), (\mid U \mid -1) + (\mid S \mid -1)\} = \mid U \mid -1$ and the removal of $\mid U \mid -1$ vertices from $U$ results in a forest. Now, our claim is to prove that $F$ is minimum. On the contrary, assume that removing at least $\mid U \mid$ - 2 vertices from $U$ results in a forest. i.e., $G\backslash F$ has at least two vertices in $U$, say $v,w \in U$. Since, $G$ is $2K_2$-free $\mid P_{vw} \mid \leq 4$. Note that, $\mid P_{vw} \mid \neq 2$ because every edge in $G_1$ is between an universal vertex and a non-universal vertex in $G_1$, by \emph{Theorem \ref{c3c5}}. Similarly, $\mid P_{vw} \mid \neq 4$. Thus, the only possibility is $\mid P_{vw} \mid = 3$. Therefore, $(P_{vw},x)$ forms an induced $C_4$, which is a contradiction to $F$.    
\item[$\bullet$] $\mid S \mid > 1$. i.e., $S$ has at least two vertices, say $x, y \in V(S)$. Our claim is to prove $S$ is minimum.
\begin{itemize}
\item[-] $F = min \{ \mid U \mid + (\mid  T \mid -1), (\mid U \mid -1) + (\mid S \mid -1)\} = \mid U \mid + (\mid  T \mid -1)$ \\
On the contrary, assume that there exist a proper subset $M$ of $U \cup (T \backslash \{a\})$ with cardinality $F' < F-1$, whose removal results in a forest, where $a \in T$. Suppose if we remove a vertex $v$ from $U$ to get a subset of $M$, then for any $b \in T$, $(b,x,v,y)$ forms an induced $C_4$, which is a contradiction. Similarly, if we  remove a vertex $b$ from $T \backslash \{a\}$ to get a subset of $M$, then $(a,x,b,y)$ forms an induced $C_4$, which is a contradiction. Thus, the set $M$ is not possible and $F$ is minimum.

\item[-] $F = min \{ \mid U \mid + (\mid  T \mid -1), (\mid U \mid -1) + (\mid S \mid -1)\} = (\mid U \mid -1) + (\mid S \mid -1)$ \\
On the contrary, assume that there exist a proper subset $M$ of $(U \backslash \{v\}) \cup (S \backslash \{x\})$ with cardinality $F' < F-1$, whose removal results in a forest, where $v \in U$. Suppose if we remove a vertex $w$ from $U$ to get a subset of $M$, then for any $y \in V(S)$, $(P_{vw},y)$ forms an induced $C_4$ (by the previous subcase), which is a contradiction. Similarly, if we  remove a vertex $y$ from $V(S) \backslash \{x\}$ to get a subset of $M$, then for any $a \in T$, $(a,x,v,y)$ forms an induced $C_4$, which is a contradiction. Thus, the set $M$ is not possible and $F$ is minimum.
\end{itemize}

\end{itemize}
Hence, $F$ is a minimum feedback vertex set if $G\backslash S$ has a non-trivial component, $G_1$ and if $G_1$ has at least one cycle.
\end{itemize}
From all the above cases, it is proved that $F$ is a minimum feedback vertex set. Hence, the theorem. $\hfill \qed$
\end{proof}

\emph{Theorem \ref{fvsc3c5}} naturally yields an algorithm to find the minimum feedback vertex set  in $O(n)$ time, which is linear in the input size.

\section{Conclusion}
We have presented many structural observations of $2K_2$-free graphs in terms of minimal vertex separator and using these results we have solved some of the unsolved combinatorial problems in the literature of $2K_2$-free graphs. The structural observations of the remaining four subclasses of $2K_2$-free graphs are to be explored, to check the status of feedback vertex set problem in $2K_2$-free graphs.

\end{document}